\documentclass[12pt,reqno]{amsart}
\usepackage{amsmath,amssymb,amsthm,pinlabel,tikz}
\usepackage[final,
colorlinks=true,
pdftitle={When does a RAAG split over Z?},
pdfauthor={Matt Clay}]{hyperref}

\newcounter{alphathm}

\theoremstyle{plain}
\newtheorem{theorem}{Theorem}[section]
\newtheorem{alphatheorem}[alphathm]{Theorem}
\newtheorem{lemma}[theorem]{Lemma}
\newtheorem{proposition}[theorem]{Proposition}
\newtheorem{corollary}[theorem]{Corollary}

\newtheorem*{claim*}{Claim}

\newcommand{\fakeenv}{} 
{ 
 \renewcommand{\fakeenv}{#2} 
 \theoremstyle{plain} 
 \newtheorem*{\fakeenv}{#1~\ref{#2}} 
 \begin{\fakeenv}
}
{
 \end{\fakeenv}
}

\theoremstyle{definition}
\newtheorem{definition}[theorem]{Definition}
\newtheorem{example}[theorem]{Example}

\newtheorem{remark}[theorem]{Remark}

\newcommand{\bF}{\mathbf{F}}

\newcommand{\ZZ}{\mathbb{Z}} 

\newcommand{\calA}{\mathcal{A}}

\newcommand{\calG}{\mathcal{G}}
\newcommand{\calH}{\mathcal{H}}
\newcommand{\calJ}{\mathcal{J}}

\newcommand{\la}{\langle}
\newcommand{\ra}{\rangle}

\begin{document}


\title[When does a RAAG split over $\ZZ$?]{When does a right-angled
  Artin group split over $\ZZ$?}

\author[M.~Clay]{Matt Clay}
\address{Department of Mathematical Sciences \\
  University of Arkansas\\
  Fayetteville, AR 72701}
\email{\href{mailto:mattclay@uark.edu}{mattclay@uark.edu}}

\begin{abstract}
  We show that a right-angled Artin group, defined by a graph $\Gamma$
  that has at least three vertices, does not split over an infinite
  cyclic subgroup if and only if $\Gamma$ is biconnected.  Further, we
  compute JSJ--decompositions of 1--ended right-angled Artin groups
  over infinite cyclic subgroups.
\end{abstract}

\date{\today}

\maketitle


\section{Introduction}\label{sec:intro}

Given a finite simplicial graph $\Gamma$, the right-angled Artin group
(RAAG) $A(\Gamma)$ is the group with generating set $\Gamma^0$, the
vertices of $\Gamma$, and with relations $[v,w] = 1$ whenever vertices
$v$ and $w$ span an edge in $\Gamma$.  That is:
\begin{equation*}
  A(\Gamma) = \la \, \Gamma^0 \mid [v,w] = 1 \, \forall v,w \in \Gamma^0
  \mbox{ that span an edge in } \Gamma \, \ra
\end{equation*}

Right-angled Artin groups, simple to define, are at the focal point of
many recent developments in low-dimensional topology and geometric
group theory.  This is in part due to the richness of their subgroups,
in part due to their interpretation as an interpolation between free
groups and free abelian groups and also in part due to the frequency
at which they arise as subgroups of geometrically defined groups.
Recent work of Agol, Wise and Haglund in regards to the Virtual Haken
Conjecture show a deep relationship between 3--manifold groups and
right-angled Artin
groups~\cite{ar:Agol13,ar:HW08,ar:HW12,un:Wise,bk:Wise12}.

One of the results of this paper computes JSJ--decompositions for
1--ended right-angled Artin groups.  This decomposition is a special
type of graph of groups decomposition over infinite cyclic subgroups,
generalizing to the setting of finitely presented groups a tool from
the theory of 3--manifolds. So to begin, we are first concerned with
understanding when a right-angled Artin group splits over an infinite
cyclic subgroup.  Recall, a group $G$ \emph{splits} over a subgroup
$Z$ if $G$ can be decomposed as an amalgamated free product $G =
A\ast_Z B$ with $A \neq Z \neq B$ or as an HNN-extension $G = A\ast_Z$.

Suppose $\Gamma$ is a finite simplicial graph.  A subgraph $\Gamma_1
\subseteq \Gamma$ is \emph{induced} if two vertices of $\Gamma_1$ span an
edge in $\Gamma_1$ whenever they span an edge in $\Gamma$.  If
$\Gamma_1 \subseteq \Gamma$ is a induced subgraph, then the natural map
induced by subgraph inclusion $A(\Gamma_1) \to A(\Gamma)$ is
injective.  A vertex $v \in \Gamma^0$ is a \emph{cut vertex} if the
induced subgraph spanned by the vertices $\Gamma^0 - \{v\}$ has more
connected components than $\Gamma$.  A graph $\Gamma$ is
\emph{biconnected} if for each vertex $v \in \Gamma^0$, the induced
subgraph spanned by the vertices $\Gamma^0 - \{v\}$ is connected.  In
other words, $\Gamma$ is biconnected if $\Gamma$ is connected and does not contain a cut vertex.
Note, $K_2$, the complete graph on two vertices, is biconnected.

\begin{remark}\label{rem:cut vertex}
  There is an obvious sufficient condition for a right-angled Artin
  group to split over a subgroup isomorphic to $\ZZ$. (In what follows
  we will abuse notation and simply say that the group splits over
  $\ZZ$.)  Namely, if a finite simplicial graph $\Gamma$ contains two
  proper induced subgraphs $\Gamma_1,\Gamma_2 \subset \Gamma$ such that
  $\Gamma_1 \cup \Gamma_2 = \Gamma$ and $\Gamma_1 \cap \Gamma_2 = v
  \in \Gamma^0$, then $A(\Gamma)$ splits over $\ZZ$.  Indeed, in this
  case we have $A(\Gamma) = A(\Gamma_1) \ast_{A(v)} A(\Gamma_2)$.

  If $\Gamma$ has at least three vertices, such subgraphs exist if and
  only if $\Gamma$ is disconnected or has a cut vertex, i.e., $\Gamma$
  is not biconnected.
\end{remark}

Our first theorem, proved in Section~\ref{sec:split}, states that this
condition is necessary as well.

\begin{alphatheorem}[$\ZZ$--splittings of RAAGs]\label{thmA}
  Suppose $\Gamma$ is a finite simplicial graph that has at least three
  vertices.  Then $\Gamma$ is biconnected if and only if $A(\Gamma)$
  does not split over $\ZZ$.
\end{alphatheorem}

If $\Gamma$ has one vertex, then $A(\Gamma) \cong \ZZ$, which does not
split over $\ZZ$.  If $\Gamma$ has two vertices, then $A(\Gamma) \cong
F_2$ or $A(\Gamma) \cong \ZZ^2$, both of which do split over $\ZZ$ as
HNN-extensions.

\begin{remark}
\label{rem:1-ended}
We recall for the reader the characterization of splittings of right-angled Artin groups over the trivial subgroup.  Suppose $\Gamma$ is a finite simpicial graph with at least two vertices.  Then $\Gamma$ is connected if and only if $A(\Gamma)$ is freely indecomposable, equivalently 1--ended.  See for instance~\cite{ar:BM01}.      
\end{remark}

In Section~\ref{sec:jsj}, for 1--ended right-angled Artin groups
$A(\Gamma)$ we describe a certain graph of groups decomposition,
$\calJ(\Gamma)$, with infinite cyclic edge groups.  The base graph for
$\calJ(\Gamma)$ is defined by considering the biconnected components
of $\Gamma$, taking special care with the $K_2$ components that
contain a valence one vertex from the original graph $\Gamma$.  Our
second theorem shows that this decomposition is a JSJ--decomposition.

\begin{alphatheorem}[JSJ--decompositions of RAAGs]\label{thmB}
  Suppose $\Gamma$ is a connected finite simplicial graph that has at
  least three vertices.  Then $\calJ(\Gamma)$ is a
  JSJ--decomposition for $A(\Gamma)$.
\end{alphatheorem}

\subsection*{Acknowledgements} Thanks go to Matthew Day for posing the
questions that led to this work.  Also I thank Vincent Guirardel and Gilbert
Levitt for suggesting the use of their formulation of a
JSJ--decomposition which led to a simplification of the exposition in
Section~\ref{sec:jsj}.  Finally, I thank Denis Ovchinnikov and the anonymous referee for noticing an error in a previous version in the proof of Proposition~\ref{prop:biconnected} which resulted in a simplification in the proof of Theorem~\ref{thmA}.


\section{Splittings of RAAGs over $\ZZ$}\label{sec:split}

This section contains the proof of Theorem~\ref{thmA}.  The outline is
as follows.  First, we will exhibit a family of right-angled Artin
groups that do not split over $\ZZ$.  Then we will show how if
$A(\Gamma)$ is sufficiently covered by subgroups that do not split
over $\ZZ$, then neither does $A(\Gamma)$.  Finally, we will show how to
find enough subgroups to sufficiently cover $A(\Gamma)$ when $\Gamma$
has at least three vertices and is biconnected.


\subsection*{Property $\bF(\calH)$} We begin by recalling some basic notions about group actions on trees, see \cite{bk:Serre03} for proofs.  In what follows, all trees are simplicial and all actions are without inversions, that is $ge \neq \bar{e}$ for all $g \in G$ and edges $e$.  When a group $G$ acts on a tree $T$, the \emph{length} of an element $g \in G$ is $|g| = \inf \{ d_{T}(x,gx) \mid x \in T \}$ and the 
\emph{characteristic
  subtree} is $T_g = \{ x \in T \mid d_{T}(x,gx) = |g| \}$.  The characteristic subtree is always non-empty.  If $|g| = 0$, then $g$ is said to be \emph{elliptic} and $T_{g}$ consists of the set of fixed points.  Else, $|g| > 0$ and $g$ is said to be \emph{hyperbolic}, in which case $T_{g}$ is a linear subtree, called the \emph{axis} of $g$, and $g$ acts on $T_{g}$ as a translation by $|g|$.
  
The following property puts some
control over the subgroups that a given group can split over.

\begin{definition}\label{def:FH}
  Suppose $\calH$ is a collection of groups.  We say a group $G$
  \emph{has property} $\bF(\calH)$ if whenever $G$ acts on a tree,
  then either there is a global fixed point or $G$ has a subgroup
  isomorphic to some group in $\calH$ that fixes an edge.

  If $\calH = \{ H \}$ we will write $\bF(H)$.
\end{definition}

\begin{remark}\label{rem:FH}
  Bass--Serre theory~\cite{bk:Serre03} implies that if $G$ has
  property $\bF(\calH)$ and $G$ splits over a subgroup $Z$, then $Z$
  has a subgroup isomorphic to some group in $\calH$.
\end{remark}

For the sequel we consider the collection $\calH = \{ F_2,\ZZ^2\}$,
where $F_2$ is the free group of rank 2.  We can reformulate the
question posed in the title using the following proposition.

\begin{proposition}\label{prop:Fz2f2}
  Suppose $\Gamma$ is a finite simplicial graph that has at least three
  vertices.  Then $A(\Gamma)$ has property $\bF(\calH)$ if and only
  if $A(\Gamma)$ does not split over $\ZZ$.
\end{proposition}

\begin{proof}
  Bass--Serre theory (Remark~\ref{rem:FH}) implies that if $A(\Gamma)$ has
  property $\bF(\calH)$ then $A(\Gamma)$ does not split over $\ZZ$.

  Conversely, suppose that $A(\Gamma)$ does not split over $\ZZ$ and $A(\Gamma)$ acts
  on a tree $T$ without a global fixed point.  The stabilizer of any edge is non-trivial as freely decomposable right-angled Artin groups whose defining graphs have at least three vertices split over $\ZZ$ (Remarks~\ref{rem:cut vertex} and~\ref{rem:1-ended}).
  
    We claim the stabilizer
  of any edge contains two elements that do not generate a cyclic
  group.  As a subgroup generated by two elements in a right-angled
  Artin group is either abelian or isomorphic to $F_2$
  \cite{ar:Baudisch81}, this shows that $A(\Gamma)$ has property $\bF(\calH)$.
  To prove the claim, let $Z$ denote the stabilizer of some edge of
  $T$ and suppose $\la g, h \ra \cong \ZZ$ for all $g,h \in Z$.  Thus
  $Z$ is abelian.  Since abelian subgroups of right-angled Artin groups
  are finitely generated (as the Salvetti complex is a finite $K(A(\Gamma),1)$~\cite{col:CD95}) we have $Z \cong \ZZ$.  But this contradicts
  our assumption that $A(\Gamma)$ does not split over $\ZZ$.
\end{proof}

Thus we are reduced to proving that property $\bF(\calH)$ is
equivalent to biconnectivity for right-angled Artin groups whose
defining graph has at least three vertices.


\subsection*{A family of right-angled Artin groups that do not split over $\ZZ$}

The following simple lemma of Culler--Vogtmann relates the
characteristic subtrees of commuting elements.  As the proof is short,
we reproduce it here.

\begin{lemma}[{Culler--Vogtmann~\cite[Lemma~1.1]{ar:CV96}}]\label{lem:cv}
  Suppose a group $G$ acts on a tree $T$ and let $g$ and $h$ be
  commuting elements.  Then the characteristic subtree of $g$ is
  invariant under $h$.  In particular, if $h$ is hyperbolic, then the
  characteristic subtree of $g$ contains $T_h$.
\end{lemma}

\begin{proof}
  As $h(T_g) = T_{hgh^{-1}}$ if $g$ and $h$ commute then $h(T_g) =
  T_g$.  If $h$ is hyperbolic, then every $h$--invariant subtree
  contains $T_h$.
\end{proof}

\begin{corollary}
\label{cor:cv}
If $\ZZ^{2}$ acts on a tree without a global fixed point, then for any basis $\{g,h\}$, one of the elements must act hyperbolically.
\end{corollary}

\begin{proof}
Suppose that both $g$ and $h$ are elliptic.  
As $hT_{h} = T_{h}$ and $hT_{g} = T_{g}$ by Lemma~\ref{lem:cv}, the unique segment connecting $T_{g}$ to $T_{h}$ is fixed by $h$ and hence contained in $T_{h}$.  In other words $T_{g} \cap T_{h} \neq \emptyset$ and therefore there is a global fixed point.
\end{proof}

Recall that a \emph{Hamiltonian cycle} in a graph is an embedded cycle that visits each vertex exactly once.    

\begin{lemma}
\label{lem:Hamiltonian}
If $\Gamma$ is a finite simplicial graph with at least three vertices that contains a Hamiltonian cycle, then $A(\Gamma)$ has property $\bF(\calH)$.
\end{lemma}

\begin{proof}
Enumerate the vertices of $\Gamma$ cyclically along the Hamiltonian cycle by
$v_1,\ldots,v_{n}$.  Notice that $G_i = \la v_i,v_{i+1} \ra \cong \ZZ^2$ for all $1 \leq i \leq n$ where the indices are taken modulo $n$.  

Suppose that $A(\Gamma)$ acts on a tree $T$ without a global fixed point.  Further suppose that $G_i$ does not fix an edge, for all $1 \leq i \leq n$.  

There are now two cases.

\medskip \noindent {\it Case I\textup{:} Each $G_i$ fixes a point.}  The point fixed by $G_i$ is unique as $G_i$ does not fix an edge, denote it $p_i$.  If the points $p_i$ are all the same, then there is a global fixed point, contrary to the hypothesis. Consider the subtree $S \subset T$ spanned by the $p_i$.  Let $p$ be an extremal vertex of $S$.  There is a non-empty proper subset $P \subset \{1,\ldots,n\}$ such that $p = p_i$ if and only if $i \in P$.  Let
$i_1,j_0 \in P$ be such that the indices $i_0 = i_1 - 1 \mod n$ and $j_1 = j_0 + 1 \mod n$ do not lie in $P$.  See Figure~\ref{fig:Cn-I}.  It is possible that $i_1 = j_0$ or $i_0 = j_1$.

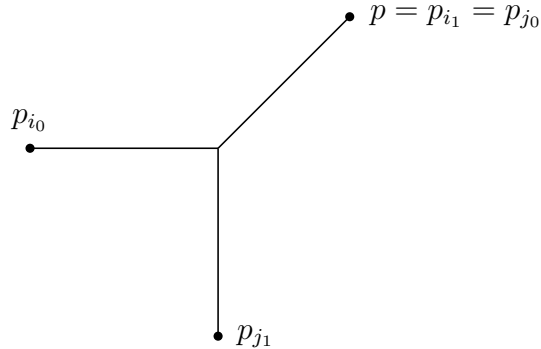
\begin{figure}[ht!]
  \begin{tikzpicture}[semithick,scale=2.5]
    \filldraw[fill=black] (0.7,0.7) circle [radius=0.02];
    \filldraw[fill=black] (-1,0) circle [radius=0.02];
    \filldraw[fill=black] (0,-1) circle [radius=0.02]; 
    \draw (0.75,0.7) node[right] {$p = p_{i_1} = p_{j_0}$}; 
    \draw (-1,0.15) node {$p_{i_0}$}; 
    \draw (0.2,-1) node {$p_{j_1}$}; 
    \draw (0.7,0.7) -- (0,0); 
    \draw (-1,0) -- (0,0); 
    \draw (0,-1) -- (0,0);
  \end{tikzpicture}
  \caption{A portion of the subtree $S \subset T$ in Case I of
    Lemma~\ref{lem:Hamiltonian}.}\label{fig:Cn-I}
\end{figure}

The element $v_{i_1} \in G_{i_0} \cap G_{i_1}$ stabilizes the non-degenerate segment $[p,p_{i_0}]$ and the element $v_{j_1} \in G_{j_0} \cap G_{j_1}$ stabilizes the non-degenerate segment $[p,p_{j_1}]$.  As $p$ is extremal, these segments overlap and thus $\la v_{i_1},v_{j_1} \ra$ fixes an edge in $T$.  This subgroup is isomorphic to either $F_2$ or $\ZZ^2$.

\medskip 

\noindent {\it Case II\textup{:} Some $G_i$ does not fix a point.} Without loss of generality, we can assume that $G_1$ does not fix a point and by Corollary~\ref{cor:cv} that $v_2$ acts hyperbolically.  By Lemma~\ref{lem:cv}, $v_1$ leaves $T_{v_2}$ invariant and so there are integers $k_1,k_2$, where $k_1 \neq 0$, such that $v_1^{k_1}v_2^{k_2}$ fixes $T_{v_2}$.  Likewise there are integers $\ell_2,\ell_3$, where  $\ell_3 \neq 0$ such that $v_2^{\ell_2}v_3^{\ell_3}$ fixes $T_{v_2}$.  Hence $\la v_1^{k_1}v_2^{k_2},v_2^{\ell_2}v_3^{\ell_3} \ra$ fixes $T_{v_2}$, in particular, this subgroup fixes an edge.  This subgroup is isomorphic to either $F_{2}$ or $\ZZ^{2}$.

\medskip

In either case, we have found a subgroup isomorphic to either $F_2$ or $\ZZ^2$ that fixes an edge.  Hence $A(\Gamma)$ has property $\bF(\calH)$.
\end{proof}


\subsection*{Promoting property $\bF(\calH)$}

We now show how to promote property $\bF(\calH)$ to $A(\Gamma)$ if
enough subgroups have property $\bF(\calH)$.

\begin{proposition}
\label{prop:decompose}
Suppose $\Gamma$ is a connected finite simplicial graph with at least three vertices and suppose that there is a collection $\calG$ of induced subgraphs $\Delta \subset \Gamma$ such that:
\begin{enumerate}
\item for each $\Delta \in \calG$, $A(\Delta)$ has property
  $\bF(\calH)$, and\label{enum:FH}
\item each two edge segment of $\Gamma$ is contained in some $\Delta
  \in \calG$.\label{enum:two edge}
\end{enumerate}
Then $A(\Gamma)$ has property $\bF(\calH)$.
\end{proposition}

\begin{proof}
Suppose $A(\Gamma)$ acts on a tree $T$ without a global fixed point.

If for some $\Delta \in \calG$, the subgroup $A(\Delta)$ does not have a fixed point, then by \eqref{enum:FH}, $A(\Delta)$, and hence $A(\Gamma)$, contains a subgroup isomorphic to either $F_2$ or $\ZZ^2$ that fixes an edge.  Therefore, we assume that each $A(\Delta)$ has a fixed point.  In particular, each vertex of $\Gamma$ acts elliptically in $T$.  Also, given three vertices $u, v, w \in \Gamma^{0}$, such that $u$ and $v$ span an edge as do $v$ and $w$, the subgroup $\la u, v, w \ra$ by \eqref{enum:two edge} is contained in some $A(\Delta)$ and hence has a fixed point.  We further may assume the fixed point of such a subgroup $\la u,v,w \ra$ to be unique for else $\la u, v \ra \cong \ZZ^{2}$ fixes an edge. 
  
As there is no global fixed point, there are vertices $v, v' \in \Gamma^{0}$ that do not share a fixed point.  Consider a path from $v$ to $v'$ and enumerate the vertices along this path $v = v_{1},\ldots,v_{n} = v'$.  If for some $1 < i < n-1$, the fixed point of $\la v_{i-1},v_{i},v_{i+1} \ra$ is different from that of $\la v_{i}, v_{i+1}, v_{i+2} \ra$, then $\la v_{i},v_{i+1} \ra \cong \ZZ^{2}$ fixes an edge as this subgroup stabilizes the non-degenerate segment between the fixed points.  If the fixed points are all the same then $v$ and $v'$ have a common fixed point, contrary to our assumptions.
\end{proof}


\subsection*{Proof of Theorem~\ref{thmA}}

Theorem~\ref{thmA} follows from Proposition~\ref{prop:Fz2f2} and the
following proposition.

\begin{proposition}
\label{prop:biconnected}
Suppose $\Gamma$ is a finite simplicial graph that has at least three vertices.  Then $\Gamma$ is biconnected if and only if $A(\Gamma)$ has property $\bF(\calH)$.
\end{proposition}

\begin{proof}
Suppose $\Gamma$ is biconnected.  Consider the collection $\calG$ of induced subgraphs $\Delta \subseteq \Gamma$ with at least three vertices that contain a Hamiltonian cycle.  By Lemma~\ref{lem:Hamiltonian}, each $\Delta \in \calG$ has property $\bF(\calH)$.  

Consider vertices $u,v,w \in \Gamma^0$ such that $u$ and $v$ span an edge $e$ and $v$ and $w$ span an edge $e'$.  As $\Gamma$ is biconnected, there is an edge path from $u$ to $w$ that avoids $v$.  Let $\rho$ be the shortest such path and let $\Delta$ be the induced subgraph of $\Gamma$ spanned by $v$ and vertices of $\rho$.  The cycle $e \cup e' \cup \rho$ is a Hamiltonian cycle in $\Delta$ and hence $\Delta \in \calG$.  The two edge segment $e \cup e'$ is contained in $\Delta$ by construction.

Hence using the collection $\calG$, Proposition~\ref{prop:decompose} implies that $A(\Gamma)$ has property $\bF(\calH)$.

Conversely, If $\Gamma$ is not biconnected, then $A(\Gamma)$ splits over $\ZZ$ and hence does not have property $\bF(\calH)$ (Remark~\ref{rem:cut vertex} and Proposition~\ref{prop:Fz2f2}).
\end{proof}


\section{JSJ--decompositions of 1--ended RAAGs}\label{sec:jsj}

We now turn our attention towards understanding all $\ZZ$--splittings
of a 1--ended right-angled Artin group.  These are exactly the groups
$A(\Gamma)$ with $\Gamma$ connected and having at least two vertices (Remark~\ref{rem:1-ended}).
The technical tool used for understanding splittings over some class
of subgroups are \emph{JSJ--decompositions}.  There are several
loosely equivalent formulations of the notion of a JSJ--decomposition
of a finitely presented group, originally defined in this setting and
whose existence was shown by Rips--Sela~\cite{ar:RS97}.  Alternative
accounts and extensions were provided by
Dunwoody--Sageev~\cite{ar:DS99}, Fujiwara--Papasogalu~\cite{ar:FP06}
and Guirardel--Levitt~\cite{un:GL}.  

We have chosen to use Guirardel and Levitt's formulation of a
JSJ--decomposition as it avoids many of the technical definitions
necessary for the other formulations---most of which have no real
significance in the current setting---and as it is particularly easy
to verify in the current setting. 

In this section we describe a JSJ--decomposition for a 1--ended
right-angled Artin group (Theorem~\ref{thmB}).  It is straightforward
to verify, given the arguments that follow, that the described graph
of groups decomposition is a JSJ--decomposition in the other
formulations as well.


\subsection*{JSJ--decompositions \` a la Guirardel and Levitt} The
defining property of a JSJ--decomposition is that it gives a
parametrization of all splittings of a finitely presented group $G$
over some special class of subgroups, here the subgroups considered
are infinite cyclic. The precise definition is as follows.

Suppose $\calA$ is a class of subgroups of $G$ that is closed under
taking subgroups and that is invariant under conjugation.  An
\emph{$\calA$--tree} is a tree with an action of $G$ such that every
edge stabilizer is in $\calA$.  An $\calA$--tree is \emph{universally
  elliptic} if its edge stabilizers are elliptic, i.e., have a fixed
point, in every $\calA$--tree.

\begin{definition}[{\cite[Definition~2]{un:GL}}]\label{def:jsj}
  A \emph{JSJ--tree} of $G$ over $\calA$ is a universally elliptic
  $\calA$--tree $T$ such that if $T'$ is a universally elliptic
  $\calA$--tree then there is a $G$--equivariant map $T \to T'$,
  equivalently, every vertex stabilizer of $T$ is elliptic in every
  universally elliptic $\calA$--tree.  The associated graph of group
  decomposition is called a \emph{JSJ--decomposition}.
\end{definition}

We will now describe what will be shown to be the JSJ--decomposition
of a 1--ended right-angled Artin group.

Suppose $\Gamma$ is a connected finite simplicial graph with at least
three vertices.  By $B_\Gamma$ we denote the \emph{block tree}, that
is, the bipartite tree with vertices either corresponding to cut
vertices of $\Gamma$ (black) or bicomponents of $\Gamma$, i.e.,
maximal biconnected induced subgraphs of $\Gamma$, (white) with an edge
between a black and a white vertex if the corresponding cut vertex
belongs to the bicomponent.  See Figure~\ref{fig:jsj} for some
examples.

For a black vertex $x \in B_\Gamma^0$, denote by $v_x$ the
corresponding cut vertex of $\Gamma$.  For a white vertex $x \in
B_\Gamma^0$, denote by $\Gamma_x$ the corresponding bicomponent of
$\Gamma$.  A white vertex $x \in B_\Gamma^0$ is call \emph{toral} if
$\Gamma_x \cong K_2$, the complete graph on two vertices.  A toral
vertex $x \in B_\Gamma$ that has valence one in $B_\Gamma$ is called
\emph{hanging}.

Associated to $\Gamma$ and $B_\Gamma$ is a graph of groups
decomposition of $A(\Gamma)$, denoted $\calJ_0(\Gamma)$.  The base
graph of $\calJ_0(\Gamma)$ is obtained from $B_\Gamma$ by attaching a
one-edge loop to each hanging vertex.  The vertex group of a black
vertex $x \in B_\Gamma^0$ is $G_x = A(v_x) \cong \ZZ$.  The vertex
group of a non-hanging white vertex $x \in B_\Gamma$ is $G_x =
A(\Gamma_x)$.  The vertex group of a hanging vertex $x \in B_\Gamma$
is $G_x = A(v)$ where $v \in \Gamma_x^0$ is the vertex that has
valence more than one in $\Gamma$.  Notice, in this latter case $v$ is
a cut vertex of $\Gamma$.  For an edge $e = [x,y] \subseteq B_\Gamma$
with $x$ black we set $G_e = A(v_x) \cong \ZZ$ with inclusion maps
given by subgraph inclusion.  If $e$ is a one-edge loop adjacent to a
hanging vertex $x$, we set $G_e = G_x$ where the two inclusion maps
are isomorphisms and the stable letter corresponding to the loop is
$w$ where $w \in \Gamma_x^0$ is the vertex that has valence one in
$\Gamma$.

By collapsing an edge adjacent to each valence two black vertex we
obtain a graph of groups decomposition of $A(\Gamma)$, which
we denote $\calJ(\Gamma)$.  It is not necessary for what follows, but we remark that the graph is groups $\calJ(\Gamma)$ is \emph{reduced \textup{(}in the sense of Bestvina--Feighn~\cite{ar:BF91}\textup{)}}, that is, for each vertex of valence less than three the edge groups are proper subgroups of the vertex group.  This property is required for a JSJ--decomposition as defined by Rips--Sela.  Observe that all edge groups of $\calJ(\Gamma)$ are of the form $A(v)$ for some vertex $v \in \Gamma^0$ and in particular maximal infinite cyclic subgroups.  By $T_{\calJ(\Gamma)}$ we denote the associated Bass--Serre tree.
 
\begin{example}\label{ex:jsj}
  Examples of $B_\Gamma$, $\calJ_0(\Gamma)$ and $\calJ(\Gamma)$ for
  two different graphs are shown in Figure~\ref{fig:jsj}.  We have
  $A(\Gamma_1) \cong F_3 \times \ZZ$.  The graph of groups
  decomposition $\calJ_0(\Gamma_1)$ is already reduced so
  $\calJ(\Gamma_1) = \calJ_0(\Gamma_1)$.  In $\calJ(\Gamma_1)$ all of
  the vertex and edge groups are infinite cyclic and all inclusion
  maps are isomorphisms.  Considering the other example,
  $\calJ(\Gamma_2)$ corresponds to the graph of groups decomposition
  $A(\Gamma_2) = \ZZ^3 \ast_\ZZ \ZZ^2 \ast_\ZZ \ZZ^3$ where the
  inclusion maps have image a primitive vector and the images in
  $\ZZ^2$ constitute a basis of $\ZZ^2$.
   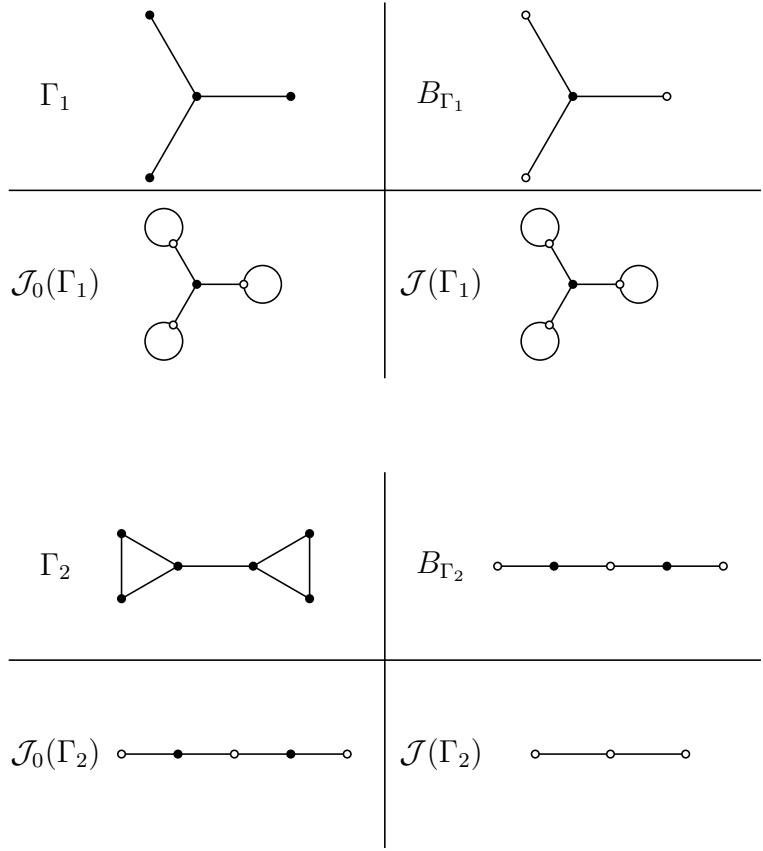
\begin{figure}[ht!]
    \begin{tikzpicture}[semithick,scale=2.5]
      \draw (-2,0) -- (2,0);
      \draw (0,-1) -- (0,1);
      \draw (-2,-2.5) -- (2,-2.5);
      \draw (0,-1.5) -- (0,-3.5);
      \draw (-1,0.5) -- (-0.5,0.5);
      \draw (-1,0.5) -- (-1.25,0.933);
      \draw (-1,0.5) -- (-1.25,0.067);
      \filldraw[fill=black] (-1,0.5) circle [radius=0.02];
      \filldraw[fill=black] (-0.5,0.5) circle [radius=0.02];
      \filldraw[fill=black] (-1.25,0.933) circle [radius=0.02];
      \filldraw[fill=black] (-1.25,0.067) circle [radius=0.02];
      \draw (-1.75,0.5) node {$\Gamma_1$}; 
      \draw (1,0.5) -- (1.5,0.5);
      \draw (1,0.5) -- (0.75,0.933);
      \draw (1,0.5) -- (0.75,0.067);
      \filldraw[fill=black] (1,0.5) circle [radius=0.02];
      \filldraw[fill=white] (1.5,0.5) circle [radius=0.02];
      \filldraw[fill=white] (0.75,0.933) circle [radius=0.02];
      \filldraw[fill=white] (0.75,0.067) circle [radius=0.02];
      \draw (0.3,0.5) node {$B_{\Gamma_1}$}; 
      \draw (-1,-0.5) -- (-0.75,-0.5);
      \draw (-0.65,-0.5) circle [radius=0.1];
      \draw (-1,-0.5) -- (-1.125,-0.7165);
      \draw (-1.175,-0.8031) circle [radius=0.1];
      \draw (-1,-0.5) -- (-1.125,-0.2835);
      \draw (-1.175,-0.1969) circle [radius=0.1];
      \filldraw[fill=black] (-1,-0.5) circle [radius=0.02];
      \filldraw[fill=white] (-0.75,-0.5) circle [radius=0.02];
      \filldraw[fill=white] (-1.125,-0.7165) circle [radius=0.02];
      \filldraw[fill=white] (-1.125,-0.2835) circle [radius=0.02];
      \draw (-1.75, -0.5) node {$\calJ_0(\Gamma_1)$};
      \draw (0.3, -0.5) node {$\calJ(\Gamma_1)$};
      \draw (1,-0.5) -- (1.25,-0.5);
      \draw (-0.65+2,-0.5) circle [radius=0.1];
      \draw (1,-0.5) -- (-1.125+2,-0.7165);
      \draw (-1.175+2,-0.8031) circle [radius=0.1];
      \draw (1,-0.5) -- (-1.125+2,-0.2835);
      \draw (-1.175+2,-0.1969) circle [radius=0.1];
      \filldraw[fill=black] (1,-0.5) circle [radius=0.02];
      \filldraw[fill=white] (1.25,-0.5) circle [radius=0.02];
      \filldraw[fill=white] (-1.125+2,-0.7165) circle [radius=0.02];
      \filldraw[fill=white] (-1.125+2,-0.2835) circle [radius=0.02];
      \draw (-1.1,-2) -- (-0.7,-2);
      \draw (-1.1,-2) -- (-1.4,-1.8268) -- (-1.4,-2.1732) -- (-1.1,-2);
      \draw (-0.7,-2) -- (-0.4,-1.8268) -- (-0.4,-2.1732) -- (-0.7,-2);
      \filldraw[fill=black] (-1.4,-1.8268) circle [radius=0.02];
      \filldraw[fill=black] (-1.4,-2.1732) circle [radius=0.02];
      \filldraw[fill=black] (-0.4,-1.8268) circle [radius=0.02];
      \filldraw[fill=black] (-0.4,-2.1732) circle [radius=0.02];
      \filldraw[fill=black] (-1.1,-2) circle [radius=0.02];
      \filldraw[fill=black] (-0.7,-2) circle [radius=0.02];
      \draw (-1.75,-2) node {$\Gamma_2$};
      \draw (0.6,-2) -- (1.8,-2);
      \filldraw[fill=white] (0.6,-2) circle [radius=0.02];
      \filldraw[fill=black] (0.9,-2) circle [radius=0.02];
      \filldraw[fill=white] (1.2,-2) circle [radius=0.02];
      \filldraw[fill=black] (1.5,-2) circle [radius=0.02];
      \filldraw[fill=white] (1.8,-2) circle [radius=0.02];
      \draw (0.3,-2) node {$B_{\Gamma_2}$};
      \draw (-1.4,-3) -- (-0.2,-3);
      \filldraw[fill=white] (-1.4,-3) circle [radius=0.02];
      \filldraw[fill=black] (-1.1,-3) circle [radius=0.02];
      \filldraw[fill=white] (-0.8,-3) circle [radius=0.02];
      \filldraw[fill=black] (-0.5,-3) circle [radius=0.02];
      \filldraw[fill=white] (-0.2,-3) circle [radius=0.02];
      \draw (-1.75,-3) node  {$\calJ_0(\Gamma_2)$};
      \draw (0.8,-3) -- (1.6,-3);
      \filldraw[fill=white] (0.8,-3) circle [radius=0.02];
      \filldraw[fill=white] (1.2,-3) circle [radius=0.02];
      \filldraw[fill=white] (1.6,-3) circle [radius=0.02];
      \draw (0.3,-3) node  {$\calJ(\Gamma_{2})$};
    \end{tikzpicture}
    \caption{Examples of $B_\Gamma$, $\calJ_0(\Gamma)$ and
      $\calJ(\Gamma)$.}\label{fig:jsj}
  \end{figure}

\end{example}

\subsection*{Proof of Theorem~\ref{thmB}} Theorem~\ref{thmB} follows
immediately from the following lemma.

\begin{lemma}\label{lem:Jvertex-elliptic}
  Suppose $\Gamma$ is a connected finite simplicial graph that has at
  least three vertices and let $\calA$ be the collection of all cyclic
  subgroups of $A(\Gamma)$.  Every vertex stabilizer of
  $T_{\calJ(\Gamma)}$ is elliptic in every $\calA$--tree.

  In particular, every edge stabilizer of $T_{\calJ(\Gamma)}$ is
  elliptic in every $\calA$--tree and so $T_{\calJ(\Gamma)}$ is
  universally elliptic and every vertex stabilizer of
  $T_{\calJ(\Gamma)}$ is elliptic in every universally elliptic
  $\calA$--tree.
\end{lemma}

\begin{proof}
  Let $T$ be an $\calA$--tree.  As $A(\Gamma)$ is 1--ended, every edge
  stabilizer of $T$ is infinite cyclic.  As the vertex groups of a
  black vertex is a subgroup of the vertex group of some white vertex,
  we only need to consider white vertices.  The vertex group of every
  non-toral vertex of $\calJ(\Gamma)$ is elliptic by
  Proposition~\ref{prop:biconnected}.

  Let $x \in B_\Gamma$ be a non-hanging toral vertex.  Denote the
  vertices of $\Gamma_x \cong K_2$ by $v_1$ and $v_2$.  Then there are
  vertices $w_1, w_2 \in \Gamma^0$ such that $[v_i, w_j] = 1$ if and
  only if $i = j$.  In other words, the vertices $w_1,v_1,v_2,w_2$
  span an induced subgraph of $\Gamma$ that is isomorphic to the path
  graph with three edges.

  If $v_1 \in G_x = A(\Gamma_x) \cong \ZZ^2$ acts hyperbolically, then
  by Lemma~\ref{lem:cv} the characteristic subtree of both $w_1$ and
  $v_2$ contains $T_{v_1}$, the axis of $v_1$.  As in the proof of
  Lemma~\ref{lem:Hamiltonian}, we find integers $k_0,k_1,\ell_0,\ell_1$ with
  $k_1, \ell_1 \neq 0$ such that $\la v_1^{k_0}w_1^{k_1},
  v_1^{\ell_0}v_2^{\ell_1} \ra \cong F_2$ fixes $T_{v_1}$ and hence
  fixes an edge.  As every edge stabilizer of $T$ is infinite cyclic,
  this shows that $v_1$ must have a fixed point.  By symmetry $v_2$
  must also have a fixed point.  Since $A(\Gamma_x) = \la v_1,v_2 \ra
  \cong \ZZ^2$, by Corollary~\ref{cor:cv} this implies that $A(\Gamma_x)$ acts elliptically.

Finally, let $x \in B_\Gamma$ be a hanging vertex.  Either $G_x$ is a subgroup of some non-hanging white vertex subgroup and so $G_x$ acts elliptically by the above argument, or $A(\Gamma) \cong F_n \times \ZZ$ for $n \geq 2$ where $G_x$ is the $\ZZ$ factor as is the case for $\Gamma_1$ in Example~\ref{ex:jsj}.  In the latter case, as $G_{x}$ is central, by Lemma~\ref{lem:cv} if $G_{x}$ acts hyperbolically, then $F_{n} \times \ZZ$ acts on its axis.  Therefore there is a homomorphism $F_n \times \ZZ \to \ZZ$ whose kernel fixes an edge.  As every edge stabilizer of $T$ is infinite cyclic, $G_x$ must act elliptically.
\end{proof}

We record the following corollary of Lemma~\ref{lem:Jvertex-elliptic}.

\begin{corollary}
\label{cor:thmB}
Suppose $\Gamma$ is a connected finite simplicial graph that has at least three vertices.  If $A(\Gamma)$ acts on a tree $T$ such that the stabilizer of every edge is infinite cyclic, then every $v \in \Gamma^{0}$ that has valence greater than one acts elliptically in $T$.
\end{corollary}

\begin{proof}
This follows from Lemma~\ref{lem:Jvertex-elliptic} as each such vertex is contained in some bicomponent $\Gamma_{x}$ for some non-hanging $x \in B_{\Gamma}$ and hence acts elliptically in $T_{\calJ(\Gamma)}$.
\end{proof}


\bibliographystyle{amsplain}
\bibliography{../bibliography}

\def\cprime{$'$}
\providecommand{\bysame}{\leavevmode\hbox to3em{\hrulefill}\thinspace}
\providecommand{\MR}{\relax\ifhmode\unskip\space\fi MR }
\providecommand{\MRhref}[2]{%
  \href{http://www.ams.org/mathscinet-getitem?mr=#1}{#2}
}
\providecommand{\href}[2]{#2}
\begin{thebibliography}{10}

\bibitem{ar:Agol13}
Ian Agol, \emph{The virtual {H}aken conjecture}, Doc. Math. \textbf{18} (2013),
  1045--1087, With an appendix by Agol, Daniel Groves, and Jason Manning.
  \MR{3104553}

\bibitem{ar:Baudisch81}
A.~Baudisch, \emph{Subgroups of semifree groups}, Acta Math. Acad. Sci. Hungar.
  \textbf{38} (1981), no.~1-4, 19--28. \MR{634562 (82k:20059)}

\bibitem{ar:BF91}
Mladen Bestvina and Mark Feighn, \emph{Bounding the complexity of simplicial
  group actions on trees}, Invent. Math. \textbf{103} (1991), no.~3, 449--469.
  \MR{1091614 (92c:20044)}

\bibitem{ar:BM01}
Noel Brady and John Meier, \emph{Connectivity at infinity for right angled
  {A}rtin groups}, Trans. Amer. Math. Soc. \textbf{353} (2001), no.~1,
  117--132. \MR{1675166 (2001b:20068)}

\bibitem{col:CD95}
Ruth Charney and Michael~W. Davis, \emph{Finite {$K(\pi, 1)$}s for {A}rtin
  groups}, Prospects in topology ({P}rinceton, {NJ}, 1994), Ann. of Math.
  Stud., vol. 138, Princeton Univ. Press, Princeton, NJ, 1995, pp.~110--124.
  \MR{1368655 (97a:57001)}

\bibitem{ar:CV96}
Marc Culler and Karen Vogtmann, \emph{A group-theoretic criterion for property
  {${\rm FA}$}}, Proc. Amer. Math. Soc. \textbf{124} (1996), no.~3, 677--683.
  \MR{1307506 (96f:20040)}

\bibitem{ar:DS99}
M.~J. Dunwoody and M.~E. Sageev, \emph{J{SJ}-splittings for finitely presented
  groups over slender groups}, Invent. Math. \textbf{135} (1999), no.~1,
  25--44. \MR{MR1664694 (2000b:20050)}

\bibitem{ar:FP06}
K.~Fujiwara and P.~Papasoglu, \emph{J{SJ}-decompositions of finitely presented
  groups and complexes of groups}, Geom. Funct. Anal. \textbf{16} (2006),
  no.~1, 70--125. \MR{MR2221253}

\bibitem{un:GL}
Vincent Guirardel and Gilbert Levitt, \emph{{JSJ} decompositions: definitions,
  existence, uniqueness. {I}: The {JSJ} deformation space}, {P}reprint,
  \href{http://arxiv.org/abs/0911.3173}{arXiv:math/0911.3173}, 2010.

\bibitem{ar:HW08}
Fr{\'e}d{\'e}ric Haglund and Daniel~T. Wise, \emph{Special cube complexes},
  Geom. Funct. Anal. \textbf{17} (2008), no.~5, 1551--1620. \MR{2377497
  (2009a:20061)}

\bibitem{ar:HW12}
\bysame, \emph{A combination theorem for special cube complexes}, Ann. of Math.
  (2) \textbf{176} (2012), no.~3, 1427--1482. \MR{2979855}

\bibitem{ar:RS97}
E.~Rips and Z.~Sela, \emph{Cyclic splittings of finitely presented groups and
  the canonical {JSJ} decomposition}, Ann. of Math. (2) \textbf{146} (1997),
  no.~1, 53--109. \MR{MR1469317 (98m:20044)}

\bibitem{bk:Serre03}
Jean-Pierre Serre, \emph{Trees}, Springer Monographs in Mathematics,
  Springer-Verlag, Berlin, 2003, Translated from the French original by John
  Stillwell, Corrected 2nd printing of the 1980 English translation.
  \MR{MR1954121 (2003m:20032)}

\bibitem{un:Wise}
Daniel~T. Wise, \emph{The structure of groups with a quasiconvex hierarchy},
  {P}reprint.

\bibitem{bk:Wise12}
\bysame, \emph{From riches to raags: 3-manifolds, right-angled {A}rtin groups,
  and cubical geometry}, CBMS Regional Conference Series in Mathematics, vol.
  117, Published for the Conference Board of the Mathematical Sciences,
  Washington, DC, 2012. \MR{2986461}

\end{thebibliography}

\end{document}